\documentclass{preprint}

\Title{Kolmogorov Equations for Randomly Perturbed Generalized Newtonian Fluids}
\ShortTitle{Kolmogorov Equations for Randomly Perturbed Generalized Newtonian Fluids}
\Author{Martin Sauer} 
\Address{Institut f\"ur Mathematik, Technische Universit\"at Berlin, Stra\ss e des 17. Juni 136, D-10623 Berlin, Germany. E-mail address: \texttt{sauer@math.tu-berlin.de}.}
\Abstract{We consider incompressible generalized Newtonian fluids in two space dimensions perturbed by an additive Gaussian noise. The velocity field of such a fluid obeys a stochastic partial differential equation with fully nonlinear drift due to the dependence of 
viscosity on the shear rate. In particular, we assume that the extra stress tensor is of power law type, i.\,e. a polynomial of degree $p-1$, $p \in (1,2)$, i.\,e. the shear thinning case. We prove that the associated Kolmogorov operator $K$ admits at least one infinitesimally invariant measure $\mu$ satisfying certain exponential moment estimates. Moreover, $K$ is $L^2$-unique w.\,r.\,t. $\mu$ provided $p \in (p^\ast,2)$, where $p^\ast$ is the second root of $p^3 - 8p^2 + 14p -6 =0$, approximately $p^\ast \approx 1.60407$.}
\AMSsub{60H15, 75A05, 76M35, 76D05}

\newcommand{\opA}{\mathcal{A}_{p}}
\newcommand{\opAn}{\mathcal{A}_{p,n}}
\newcommand{\symg}{\mathbf{E}u}
\newcommand{\To}{{\mathbb T}}
\newcommand{\unt}{u_n(t,x)}
\newcommand{\uns}{u_n(s,x)}
\newcommand{\Ip}{\mathcal{I}_p}
\newcommand{\zstar}{\Z^2_\ast}
\newcommand{\zplus}{\Z^2_+}
\newcommand{\fcb}{\mathcal{F}C_b^2}

\begin{document}
\section{Introduction}
Kolmogorov equations in infinite dimensions have been worked on quite intensively, especially during the last decade. In mathematical fluid dynamics the efforts mostly concentrated on the stochastic Navier-Stokes equations in two space dimensions. We refer to Flandoli and Gozzi \cite{FlandoliGozzi}, Da Prato et. al. \cite{BarbuChannel, BarbuNSE1, BarbuNSE2, BarbuReflection, DPDBurgers} and Stannat \cite{Stannat2DNSE, StaCoriolis} for articles on stochastic Burgers and Navier-Stokes equations on different domains and with different choices of noise regularity exploiting a variety of strategies to prove $L^p$-uniqueness of the associated Kolmogorov operators w.\,r.\,t. some infinitesimally invariant measure. In this article we extend such results to a more general and likewise more realistic fluid model. In detail, we consider a homogeneously distributed, incompressible fluid under the influence of additive Gaussian random forces. The so-called stress tensor $\mathbf{S}$ modeling molecular forces inside the fluid, in particular shear stress, is 
assumed to have a nonlinear dependence on the 
symmetric part of the velocity gradient $\symg \df \frac12 ( \nabla u + \nabla u^T)$. One often assumes a functional structure of the form $\mathbf{S}(\symg) = \nu (\abs{\symg}) \symg$ with a viscosity function $\nu \in C([0,\infty))$ depending on the Frobenius or Hilbert-Schmidt norm $\abs{\symg}^2 \df (\sum_{ij} (\symg)_{ij}^2)$. Such fluids are usually referred to as generalized Newtonian. For the precise assumptions on $\mathbf{S}$ we refer to the next section. The equations concerning conservation of mass and momentum then read as follows.
\begin{empheq}[right=\quad\empheqrbrace]{equation}\label{eq:SPDE}
\begin{aligned}
\partial_t u &= \Div \mathbf{S} (\symg ) - \big( (u \cdot \nabla ) u\big) - \nabla \pi + \eta & \text{in } [0,\infty) \times \To^2,\\
\Div u &= 0 &\text{in }[0,\infty) \times \To^2,\\
u(0)&= u_0  &\text{in } \To^2,
\end{aligned}
\end{empheq}
These describe the time evolution of the velocity field $u$ and the hydrodynamic pressure $\pi$. For technical reasons we restrict ourselves to the space periodic setting on the two-dimensional torus $\To^2 \df (0,2\pi)^2$ with periodic boundary conditions. $\eta$ is supposed to be some Gaussian noise, white in time and colored in space. Recall that the choice $\nu(x) = \nu_0 x$, i.\,e. the linear case $\mathbf{S}(\symg) = \nu_0 \symg$, just reduces the equations to the well-known stochastic Navier-Stokes equations.

The case of $\eta$ being some given deterministic external force has been under consideration in the PDE literature for more than four decades. We refer to Ladyzhenskaya \cite{Lady1} and Lions \cite{Lions} for the pioneering contributions in this direction. Extensions of these results for such generalized Newtonian fluid models in the space periodic setting can be found in the monograph be Necas et. al. \cite{NecasBook}, also note \cite{MalekRajagopal} for a more recent overview. In the SPDE community these models have attracted attention only recently, see Terasawa and Yoshida \cite{TerasawaPLF} for existence and uniqueness results on weak martingale solutions as well as Liu and R\"ockner \cite{LiuLocallyMonotone} for variational solutions. The extra stress tensor chosen there is a polynomial of degree $p-1$ and the corresponding viscosity function is given by
\begin{equation}\label{eq:viscosity}
\nu(x) = \nu_0 (1+x^2)^{\frac{p-2}{2}}, \quad p>1.
\end{equation}
In two space dimensions the condition on $p$ for the existence of weak martingale solutions is $p> 6/5$ whereas pathwise uniqueness holds provided $p \geq 2$, see \cite[Theorem 2.1.3, Theorem 2.2.1]{TerasawaPLF}. Note that \eqref{eq:SPDE} has a fully nonlinear drift (in the case $p \neq 2$) that complicates the analysis compared to equations of semi-
linear structure.

Instead of solving the SPDE \eqref{eq:SPDE} we study the associated Kolmogorov (backward) equation. For this reason consider the well-known translation of \eqref{eq:SPDE} to an abstract stochastic evolution equation on some function space $H$ given by
\begin{empheq}[right=\quad\empheqrbrace]{equation}\label{eq:abstractSEE}
\begin{aligned}
\mathrm{d} u(t) &= \Big( \opA\big(u(t)\big) + B\big( u(t) \big) \Big) \dt + \sqrt{Q} \dwt, & t \geq 0,\\
u(0) &= u_0.
\end{aligned}
\end{empheq}
Equation \eqref{eq:abstractSEE} is driven by some cylindrical Wiener process $W$ on $H$. The appropriate definitions of spaces and operators will be given in the next section. Supposed there exists a weak solution to \eqref{eq:abstractSEE}, then It\^o's formula implies that its transtion probabilities solve the Fokker-Planck equation given by the associated Kolmogorov operator. On smooth cylinder functions $\fcb$ (see next section) this operator is formally identified as
\begin{equation*}
\big(K \phi \big) (x) \df \frac12 \tr QD^2 \phi(x)+ \scp{\opA(x) + B(x)}{D \phi(x)}_H.
\end{equation*}
The Fokker-Planck equation itself -- since an equation for measures -- has not been in the center of attention, however see \cite{Bogachev1, Bogachev2} for results on existence and uniqueness in infinite dimensional Hilbert spaces. It is more common to study the associated Kolmogorov equation $\dot{v}(t,x) = Kv(t,x)$, $v(0,\cdot) = \phi$, $x \in H$ where $\phi$ is an inital condition for this parabolic PDE in infinitely many variables. Note that this Cauchy problem is only reasonable for some extension of $K$ since even if $\phi \in \fcb$ we essentially never get $v(t,\cdot) \in \fcb$. Thus, the main interest is the well-posedness of this Cauchy problem, i.\,e. the question is wether $(K,\fcb)$ can be uniquely extended to the generator of a Markovian $C_0$-semigroup on a suitable function space. Except the up to the best of our knowledge only contribution to solve this equation directly in spaces of weighted weakly continuous functions in \cite{Sobol}, the common method involves a reduction of the problem to $L^p(\mu)$, in our case $L^2(\mu)$, where $\mu$ is an infinitesimally invariant measure for $K$. Well-posedness of the Cauchy problem is equivalent to the essential maximal dissipativity of $K$ on $L^2(\mu)$, also called $L^2$-uniqueness w.\,r.\,t. the reference measure $\mu$. This may then be used to construct weak solutions in the sense of the martingale problem initiated by Stroock and Varadhan \cite{Stroock}, see e.\,g. the monograph \cite{Eberle} for a discussion on this subject.

In this article, we restrict ourselves to the power law type fluids with viscosity function \eqref{eq:viscosity} studied in \cite{TerasawaPLF}, however we are only able to study the Kolmogorov operator associated to \eqref{eq:abstractSEE} for $p \in (1,2)$ where the well-posedness of \eqref{eq:abstractSEE} and in particular uniqueness remains an open question, see the remark following Theorem \ref{thm:Measure}. Let us remark at this point, that $p=2$ can be included, i.\,e. the stochastic Navier-Stokes equations, however this particular example has already been treated in e.\,g. \cite{BarbuNSE1, StaCoriolis}. Thus, we exclude $p=2$ which considerably simplifies the presentation. We prove that for this range of $p$, $K$ admits at least one infinitesimally invariant measure, see Theorem \ref{thm:Measure}. Moreover, $K$ is $L^2$-unique under additional assumptions on $p$, see Theorem \ref{thm:MainTheorem}. The proof is based on exponential moment estimates for the infinitesimally invariant measure in Theorem \ref{thm:Measure} and pointwise gradient estimates for 
the transition semigroup of suitable finite dimensional approximating problems.  
\section{Main Results}
Before we state the main result, let us introduce some notation. By $C^\infty_{\text{per}}(\To^d; \R^2)$ we denote the space of smooth and $2\pi$-periodic vector fields on the two-dimensional torus. Furthermore let
\[
\textstyle \mathcal{D}_\sigma \df \Big\{ u \in C^\infty_{\text{per}}(\To^d; \R^2) : \Div u=0, \int u(\xi) \dxi = 0 \Big\}
\]
be the subspace of divergence free and centered vector fields. Then, the state space $H$ of equation \eqref{eq:abstractSEE} is given by the closure of $\mathcal{D}_\sigma$ in $L^2(\To^2;\R^2)$. Due to the nonlinear structure of the equation, we also need $L^p$-spaces with $p \neq 2$. Thus, introduce $(H^{k,p}, \norm{\cdot}_{k,p})$, $k\geq 1$, $p \in (1,\infty)$ as the closure of $\mathcal{D}_\sigma$ w.\,r.\,t. the standard Sobolev norms in $W^{k,p} (\To^2; \R^2)$, as well as their dual spaces $(H^{k,p})^\ast = H^{-k,q}$, $1/p + 1/q = 1$, corresponding to the closure in $W^{-k,q}(\To^d;\R^2)$. For abbreviation introduce $V\df H^{1,2}$.

As mentioned in the introduction, we study equation \eqref{eq:SPDE} under the same assumption on the extra stress tensor as in \cite{TerasawaPLF}.
\begin{assum}\label{assum:PowerLaw}
The extra stress tensor is given by $\mathbf{S}(\symg) \df \nu(\abs{\symg}) \symg$ with $\nu$ as in \eqref{eq:viscosity} with parameter $p \in (1,2)$.
\end{assum}
\begin{rem}
Note that one can impose a more general assumption on $\mathbf{S}$ using an extra stress tensor given by a so-called $p$-potential, see e.\,g. \cite{DieningRuzicka, NecasBook}. This is a common extension and is designed to mimic the main feature of the function $\nu$ in \eqref{eq:viscosity}, namely some $p$-ellipticity. However, such an extension has basically no effect on the analysis and thus we sacrifice this more general framework for a concise presentation.
\end{rem}

Let $\mathcal{P}$ denote the Helmholtz-Leray projection $\mathcal{P}: L^2(\To^2;\R^2) \to H$. Define the nonlinear operators $\opA: D(\opA) \subset H \to H, u \mapsto \mathcal{P} \Div \mathbf{S} (\symg )$, $B: D(B) \subset H \times H \to H, (u,v) \mapsto -\mathcal{P} (u \cdot \nabla) v$, $B(u) \df B(u,u)$ and the Stokes operator $A: H^{2,2} \to H, u \mapsto \mathcal{P}\Delta u$. A complete orthonormal system of eigenfunctions of $A$ corresponding to eigenvalues $- \abs{k}^2$, $k\in \zstar \df \Z^2\setminus \{0\}$ is given by
\[
e_k(\xi) \df \frac{1}{\sqrt{2}\pi} \frac{k^\bot}{\abs{k}}
\begin{cases}
\sin (k \xi) &\text{if } k \in \zplus,\\
\cos (k \xi)& \text{if } k \in -\zplus,
\end{cases}
\]
where $\zplus \df \{ k=(k_1,k_2) \in \zstar: k_1 \geq 0\}$. With this orthonormal system, the cylindrical Wiener process $W$ in \eqref{eq:abstractSEE} has a representation as the formal sum 
\[
W(t) \df \sum_{k \in \zstar} \beta_k(t) e_k
\]
with a family of independent real valued Brownian motions $(\beta_k)_{k \in \zstar}$ on a common underlying stochastic basis $(\Omega, \algF, \{\algF_t\}, \PP)$. The covariance operator $Q \in L(H)$ is symmetric, non-negative and nuclear. In this article, we shall use the following additional assumption on $Q$, which can be found in e.\,g. \cite{BarbuNSE1} in a similar formulation.
\begin{assum}\label{assum:Q}
Suppose $\sqrt{Q} \in L_2(H,V)$, i.\,e. Hilbert-Schmidt from $H$ to $V$. In particular, $\tr (-A) Q = \norm{\sqrt{Q}}_{L_2(H,V)} < \infty$ and $Q$ can be restricted to a bounded, linear operator on $V$. Moreover, $\sqrt{Q}$ can be extended to a bounded linear operator from $V^\ast$ to $H$.
\end{assum}
\begin{rem}
At this point we would like to mention that the condition on $\tr_V Q<\infty$, imposed in \cite{StaCoriolis}, is not sufficient to apply It\^o's formula to $\norm{\cdot}_V^2$. To see this, assume for simplicity that $Q$ is diagonal w.\,r.\,t. $(e_k)_{k \in \zstar}$ with eigenvalues $\sigma_k^2$. Then $\tr_H Q = \sum_{k \in \zstar} \sigma_k^2$ and
\[
\tr_V Q = \sum_{k \in \zstar} \scp{Q e_k^V}{e_k^V}_V = \sum_{k \in \zstar} \abs{k}^{-2} \scp{Q e_{-k}}{e_{-k}}_V = \sum_{k \in \zstar} \sigma_k^2 \abs{k}^{-2} \scp{e_k}{e_k}_V = \sum_{k \in \zstar} \sigma_k^2.
\]
However, $\tr_V Q < \infty$ can be a suitable condition on $Q$ in this context, provided $W$ is a cylindrical Wiener process on $V$ instead of $H$. Obviously, this allows an application of It\^o's formula to $\norm{\cdot}_V^2$.
\end{rem}
 
The solvability of \eqref{eq:abstractSEE} in the linear, i.\,e. Navier-Stokes case, has been studied in various formulations, see \cite{KuksinLecture} for a survey. In the nonlinear case, the authors in \cite{TerasawaPLF} have shown existence of weak martingale solutions under Assumption \ref{assum:Q}, $d=2$, $p> 6/5$ and $u_0 \in V$. Pathwise uniqueness holds if moreover $p \geq 2$. Let us mention for completeness that the variational approach in \cite{LiuLocallyMonotone} yields a unique strong solution for $p\geq 2$ under less assumptions on $Q$ and $u_0$, namely $\tr Q < \infty$ and $u_0 \in H$. However, there is up to now no extension to $p<2$. Now, let us get back to the associated Kolmogorov operator to \eqref{eq:abstractSEE}. In order to give meaning to its definition we need to specify a suitable domain. Recall
\begin{equation}\label{eq:KolmogorovOp}
\big(K \phi \big) (x) \df \frac12 \tr QD^2 \phi(x)+ \scp{\opA(x) + B(x)}{D \phi(x)}_H.
\end{equation}
where $D$ denotes the Fr\'ech\'et derivative w.\,r.\,t. $x$ in $H$. It is reasonable to consider the space of twice continuously differentiable and finitely based cylinder functions with bounded partial derivatives as a domain for $K$. In formula
\[
\fcb \df \Big\{ \phi(x) = \tilde{\phi}(x_{k_1}, \dots, x_{k_n}): n \in \N, \tilde{\phi} \in C_b^2\big(\R^n\big), x_{k_i} = \scp{x}{e_{k_i}}_H \Big\}.
\]
Indeed, for any $\phi \in \fcb$ the first order part and thus $K$ is well-defined.

A probability measure $\mu$ on $(H, \mathcal{B}(H))$ is called infinitesimally invariant for $(K, \fcb)$ if $K(\fcb) \subset L^1(\mu)$ and $\int K \phi \dmu = 0$ for all $\phi \in \fcb$. By general arguments $(K,\fcb)$ is well-defined on $L^1(\mu)$ and moreover dissipative, hence closable. Note that if $K(\fcb) \subset L^p(\mu)$ for some $p \geq 1$ then the same holds with the exponent $1$ replaced by $p$. For such a measure $\mu$ we are interested in the well-posedness of the Kolmogorov equation associated to \eqref{eq:abstractSEE}, i.\,e. the Cauchy problem $\dot{u} = Ku$ for $(K, \fcb)$ in the space $L^p(\mu)$. It is well-known that if the range condition
\begin{equation}\label{eq:RangeCondition}
(\lambda - K) \big(\fcb\big) \subset L^p(\mu) \text{ dense for some }\lambda >0
\end{equation}
is satisfied, the closure $(\overline{K}, D(\overline{K}))$ generates a $C_0$-semigroup $(T_t)_{t\geq 0}$ of Markovian contractions on $H$. Furthermore, this is the only $C_0$-semigroup that has a generator that extends $(K, \fcb)$. In this case we say that $(K, \fcb)$ is $L^p$-unique w.\,r.\,t. $\mu$.

Under Assumptions \ref{assum:PowerLaw} and \ref{assum:Q} we prove the following theorem on existence and moment estimates for infinitesimally invariant measures. The constant $C_p$ appearing below is the one for which inequality \eqref{proof:EstimateIp} holds.
\begin{thm}\label{thm:Measure}
There exists an infinitesimally invariant measure $\mu$ for $(K, \fcb)$ satisfying the following exponential a priori estimate
\[
\int \Ip(x) \mathrm{e}^{\eps \norm{x}_V^p} \dmu(x) < \infty, \quad \text{for }0 < \eps < \eps^\ast \fd \frac{2\nu_0(p-1)}{p \cdot C_p \norm{Q}_{L(V)}}.
\]
Here,
\[
\Ip(x) \df \int_{\To^2} \big(1+\abs{\mathbf{E}x}^2\big)^{\frac{p-2}{2}} \abs{\nabla \mathbf{E}x}^2 \dxi.
\]
Finally, $(K, \fcb)$ is well-defined on $L^2(\mu)$.
\end{thm}
\begin{rem}
This theorem also explains the reason for the restriction to $p<2$. For $(K, \fcb)$ to be well-defined on $L^1(\mu)$ the term $\scp{\opA(x)}{D\phi(x)}$ has to be $\mu$-integrable. By Lemma \ref{app:NormA} we obtain $\norm{\opA(x)}^2 \leq C \Ip(x)$ for $p < 2$, in fact also trivially for $p = 2$. However for $p > 2$ the inequality is in the opposite direction, hence we cannot conclude that $(K, \fcb)$ is well-defined on $L^1(\mu)$, although the operator $\opA$ has much more regularity and equation \eqref{eq:abstractSEE} has a unique strong solution. One may think about working in $L^p$ instead of $L^2$, but certain essential properties of $\mathbf{S}$, e.\,g. monotonicity (Lemma \ref{app:Monotone}), do not hold in $L^p$. Thus, it is an open question how to extend the results of this article to $p > 2$.
\end{rem}

\begin{assum}\label{assum:LowerBound}
Let $p \in (p^\ast,2)$ where $p^\ast$ is the second root of $p^3 -8p^2 + 14p-6 =0$, approximately $p^\ast \approx 1.60407$.
\end{assum}

Under this additional assumption the uniqueness theorem below holds.
\begin{thm}\label{thm:MainTheorem}
Let $\mu$ be any infinitesimally invariant measure satisfying the moment estimates from Theorem \ref{thm:Measure}, in particular $(K, \fcb)$ is well-defined on $L^2(\mu)$. Then the operator $(K, \fcb)$ is $L^2$-unique w.\,r.\,t. $\mu$.
\end{thm}

The proofs of Theorem \ref{thm:Measure} and \ref{thm:MainTheorem} are contained in Section 3 and 4, respectively. The overall basis are finite dimensional Galerkin approximations of \eqref{eq:abstractSEE} w.\,r.\,t. the basis $(e_k)_{k \in \zstar}$. This is in contrast to e.\,g. a smooth approximation of the convection term in \cite{BarbuNSE1} and also favorable since \eqref{eq:abstractSEE} has no semi-linear structure. Moreover, we have chosen $\fcb$ as a suitable domain for $K$, hence an approximation in coordinates rather than smoothness seems canonical. The major part of Section 3 is devoted to the uniform exponential moment estimates for the invariant measures of the approximating problems. In Section 4 we prove pointwise gradient estimates for their transition semigroups which are particularly more involved than in the Navier-Stokes case. For a clear presentation we include an appendix with all necessary technical lemmas concerning the operators $\opA$ and $B$ that may be found in other references.

\section{Proof of Theorem \ref{thm:Measure}}
We prove Theorem \ref{thm:Measure} using finite dimensional Galerkin approximations. Thus, for $n \in \N$ set $H_n \df \spann \{e_k: k \in \zstar, \abs{k} \leq n\}$ as a finite dimensional subspace of $H$ by the canonical embedding $\iota_n$. Let $\pi_n: H \to H_n$ be the projection and define $\opAn(u) \df \pi_n \opA(\iota_n u)$, $B_n(u) \df \pi_n B(\iota_n u)$ for $u\in H_n$. Of course, let $Q_n \df \pi_n \circ Q \circ \iota_n$ and $W_n(t) \df \pi_n W(t)$. The approximation of \eqref{eq:abstractSEE} is then given by
\begin{empheq}[right=\quad\empheqrbrace]{equation}\label{eq:ApproxSDE}
\begin{aligned}
\mathrm{d} u_n(t) &= \Big( \opAn\big(u_n(t)\big) + B_n\big( u_n(t) \big) \Big) \dt + \sqrt{Q_n} \dwnt, & t \geq 0,\\
u(0) &= \pi_n u_0.
\end{aligned}
\end{empheq}
Standard results on SDEs imply the existence of a unique strong solution $\unt$ to \eqref{eq:ApproxSDE}, see e.\,g. \cite[Theorem 1.2]{KrylovKolmogorov}. Moreover, this solution satisfies some a priori estimates similar to the well-known energy inequality in the Navier-Stokes case, namely
\begin{equation}
\EV{\norm{\unt}_H^2 + \int_0^t \norm{\uns}_{1,p}^p \ds} \leq C \big( \norm{u_0}_H^2 + t \tr Q \big)
\end{equation}
with a constant $C$ independent of $n$, cf. \cite[Theorem 3.1.1]{TerasawaPLF}. Due to the invariance of the enstrophy in the two dimensional space periodic setting, see Lemma \ref{app:ConvectionInvariance}, it also holds that
\begin{equation}\label{eq:APrioriEstimate}
\EV{\norm{\unt}_V^2 + \int_0^t \Ip \big(\uns\big) \ds} \leq C \big( \norm{u_0}_V^2 + t \tr (-A) Q \big),
\end{equation}
see \cite[Lemma 3.2.2]{TerasawaPLF}. These estimates imply the existence of an invariant probability measure $\mu_n$ for \eqref{eq:ApproxSDE} using the Krylov-Bogoliubov method. In addition to that, the latter inequality is exploited to deduce the desired exponential moment estimates.
\begin{prop}\label{prop:ApproxMoment}
Let $\mu_n$ be any invariant measure of \eqref{eq:ApproxSDE}. Then, for any $0 < \eps < \eps^\ast$ there exists a finite constant $C(\eps)$ independent of $n$ such that
\[
\int \Ip(x) \mathrm{e}^{\eps \norm{x}_V^p} \dmu_n(x) < C(\eps).
\]
\end{prop}
\begin{proof}
Let $u_t = \unt$ in the course of the proof and It\^o's formula for $(1 + \norm{u_t}_V^2)^m$, $m \geq 1$ implies
\begin{align*}
\big(1 + \norm{u_T}_V^2\big)^m &= \big(1 + \norm{x}_V^2\big)^m + m \int_0^T \big(1 + \norm{u_t}_V^2\big)^{m-1} \norm{\sqrt{Q_n}}_{L_2(H,V)} \dt\\
&\quad+ 2m \int_0^T \big(1 + \norm{u_t}_V^2\big)^{m-1} \scp{\opAn(u_t) + B_n(u_t)}{(-A) u_t}_H \dt\\
&\quad+ 2m (m-1) \int_0^T \big(1 + \norm{u_t}_V^2\big)^{m-2} \scp{Q_n u_t}{u_t}_V \dt\\
&\quad+ 2m \int_0^T \big(1 + \norm{u_t}_V^2\big)^{m-1} \scp{u_t}{\sqrt{Q_n} \dwnt}_V.
\end{align*}
The essential properties of the nonlinear drift can be found in Lemmas \ref{app:TestStokes} and \ref{app:ConvectionInvariance}. We can conclude that
\[
\scp{\opAn(u_t) + B_n(u_t)}{(-A) u_t}_H \leq - \nu_0 (p-1) \Ip (u_t).
\]
This nonlinear functional $\Ip$ can be estimated in terms of the $V$-norm, namely
\begin{equation}\label{proof:EstimateIp}
\norm{u_t}_V^p \leq C_p \Ip(u_t),
\end{equation}
by Lemma \ref{app:EstIp} and the embeddings $H^{2,p} \hookrightarrow V \hookrightarrow H^{1,p}$. In the following, we derive a recursive formula for monomial moments of $\mu_n$. The latter inequality suggests that the steps in the exponent are of size $p$, hence it is reasonable to expect that $\exp [ \eps \norm{x}_V^p ]$ is $\mu_n$-integrable, i.\,e. a sub-Gaussian decay. This motivates the choice $m = p/2\, k + 1$ for $k \in \N_0$.
With a standard cut-off procedure it follows that
\begin{equation}\label{proof:BasicMomentFormula}
\begin{split}
&2 \nu_0 (p-1) \int \Ip(x) \big(1 + \norm{x}_V^2\big)^{\frac{p}{2}k} \dmu_n(x)\\
&\quad\leq \bigl(\tr (-A) Q + pk\norm{Q}_{L(V)} \bigr) \int \big(1 + \norm{x}_V^2\big)^{\frac{p}{2}k} \dmu_n(x).
\end{split}
\end{equation}
The choice $k=0$ implies $\int \Ip(x) \dmu_n(x) < \infty$, hence for $k\geq 1$ use \eqref{proof:EstimateIp} to introduce $\Ip$ to the right hand side. We obtain
\begin{align*}
&\int \Ip(x) \big(1 + \norm{x}_V^2\big)^{\frac{p}{2}k} \dmu_n(x)\\
&\quad\leq \frac{C_p\bigl(\tr (-A) Q + pk\norm{Q}_{L(V)} \bigr)}{2\nu_0 (p-1)}  \int \Ip(x) \big(1 + \norm{x}_V^2\big)^{\frac{p}{2}(k-1)} \dmu_n(x).
\end{align*}
In particular, $\int \Ip(x)(1 + \norm{x}_V^2)^{\frac{p}{2}k} \dmu_n(x) \leq C(k)$ for all $k \in \N_0$ by an iteration of the inequality above. Now choose $\eps < \eps' < \eps^\ast$ such that for every $k > K_\eps$ (such an $K_\eps$ exists and is finite) it holds that
\[
\eps \frac{C_p\bigl(\tr (-A) Q + pk\norm{Q}_{L(V)} \bigr)}{2\nu_0 (p-1)} \leq  \frac{p \cdot C_p \norm{Q}_{L(V)}}{2 \nu_0 (p-1)} \eps' k = \frac{\eps'}{\eps^\ast} k
\]
by definition of $\eps^\ast$. For all $k > K_\eps$ it follows that
\[
\frac{\eps^k}{k!} \int \Ip(x) \big(1 + \norm{x}_V^2\big)^{\frac{p}{2}k} \dmu_n(x) \leq \left(\frac{\eps'}{\eps^\ast}\right)^{k-K_\eps} \frac{\eps^{K_\eps}}{K_\eps!} C(K_\eps),
\]
hence for arbitrary $K > K_\eps$ the sum
\[
\sum_{k=0}^K \frac{\eps^k}{k!} \int \Ip(x) \big(1 + \norm{x}_V^2\big)^{\frac{p}{2}k} \dmu_n(x) \leq \frac{\eps^\ast}{\eps^\ast-\eps'} \sum_{k=0}^{K_\eps} \frac{\eps^{K_\eps}}{K_\eps!} C(K_\eps)
\]
can be bounded independent of $K$ and the infinite series (as $K \to \infty$) is bounded by the same value. Note that obviously $(1+ x^2)^\frac{p}{2} > x^p$, thus
\[
\int \Ip(x) \mathrm{e}^{\eps \norm{x}_V^p} \dmu_n(x) \leq \frac{\eps^\ast}{\eps^\ast-\eps'} \sum_{k=0}^{K_\eps} \frac{\eps^{K_\eps}}{K_\eps!} C(K_\eps) \fd C(\eps). \qedhere
\]
\end{proof}
\begin{rem}
The result above generalizes to $p \geq 2$. In contrary to the case $p<2$ which lowers the decay rate, the case $p\geq 2$ simply results in a Gaussian decay. This is due to the estimate $\Ip(x) \geq C_p \norm{x}_{2,2}^2$ which is optimal, see \cite[Lemma 5.3.24]{NecasBook}. The moment estimate follows in the same line as above.
\end{rem}
\begin{proof}[Proof of Theorem \ref{thm:Measure}]
For every $n \in \N$, consider $\mu_n \in \mathcal{M}_1(H_n)$ as a probability measure $\tilde{\mu}_n$ on H by the canonical inclusion. The uniform estimate in Proposition \ref{prop:ApproxMoment} in particular guarantees that $\int \norm{x}_V^2 \, \mathrm{d} \tilde{\mu}_n(x) \leq C$ uniformly in $n$. Since the embedding $V \hookrightarrow H$ is compact, the family of measures $(\tilde{\mu}_n)_n$ is tight on $H$ by Prokhorov's theorem. Thus, we can extract a weakly converging subsequence $\tilde{\mu}_{n_k} \to \mu$. In particular, we can extend all moment estimates to the limit measure via a cut-off and Fatou's lemma. It remains to show the infinitesimal invariance. Fix $\phi \in \fcb$ and $n$ large enough such that $\phi$ only depends on coordinates $e_k, k \leq n$. Now consider the Kolmogorov operator associated with the approximating equation \eqref{eq:ApproxSDE} given by
\[
\bigl(K_n \phi\bigr) (x)\df \frac12 \tr \big(Q_n D^2\phi (x)\big) + \scp{\opAn(x) + B_n(x)}{D\phi (x)}_H.
\]
Then, the invariance of $\tilde{\mu}_n$ implies $\int (K_n \phi) (x) \,\mathrm{d}\tilde{\mu}_n(x) = 0$. Moreover, 
\[
\abs{\big(K_n \phi\big) (x)} \leq \frac12 \abs{ \tr \big(Q_n D^2 \phi (x)\big)} + \norm{\opAn(x) + B_n(x)}_H \norm{D\phi(x)}_H
\]
by the Cauchy-Schwarz inequality. The operator norms are estimated in Lemma \ref{app:NormA} and \ref{app:NormB}, more precisely
\[
\norm{\opAn(x) + B_n(x)}_H \leq C( 1 + \norm{x}_V)^{\frac{4-p}{2}} \Ip(x)^{\frac12}.
\]
The uniform moment estimates of $\tilde{\mu}_n$ now imply the convergence
\[
\lim_{k \to \infty} \int \big(K\phi\big)(x) \,\mathrm{d}\tilde{\mu}_{n_k}(x) = \int \big(K\phi\big)(x) \dmu(x).
\]
In the same way we can use this estimate for
\[
\abs{\big(( K - K_{n_k}) \phi\big) (x) } \leq \bigl( \norm{\opA(x) - \opAn(x)}_H + \norm{B(x)-B_n(x)}_H \bigr) \norm{D\phi(x)}_H
\]
and conclude
\[
\lim_{k \to \infty} \int \big((K - K_{n_k}) \phi\big)(x) \,\mathrm{d}\tilde{\mu}_{n_k}(x) = 0.
\]
All three equations combined yield $\int K \phi(x) \dmu(x) = 0$.
\end{proof}
\section{Proof of Theorem \ref{thm:MainTheorem}}
As described above, a major ingredient for the proof of such a theorem is a pointwise gradient estimate for the transition semigroup $(P_t^n)_{t \geq 0}$ induced by the unique strong solution of the approximating problem \eqref{eq:ApproxSDE}. Recall that
\[
\bigl(P_t^n \phi\bigr) (x) \df \EV{\phi \big(\unt \big)}, \quad x \in H, \phi \in \mathcal{B}_b(H_n).
\]
Regularity properties of the transition semigroup correspond to a regular dependence of $\unt$ on the initial value $x$, i.\,e. the behavior of the difference of two solutions corresponding to two different initial values. This can be illustrated by
\begin{equation}\label{eq:EstDiffTransSemi}
\abs{P_t^n \phi(x) - P_t^n \phi(y)} \leq \EV{ \abs{\phi \big(u_n(t,x)\big) - \phi \big(u_n(t,y)\big)}} \leq \norm{\phi}_{\lip} \EV{ \norm{u_n(t,x) - u_n(t,y)}_H}
\end{equation}
for $\phi \in \lip(H_n)$. Thus, the time evolution of $z_t \df u_t - v_t$, $u_t \df u_n(t,x)$ and $v_t \df u_n(t,y)$, is the crucial point. It\^o's formula implies
\begin{equation}\label{eq:DifferenceProcess}
\mathrm{d} \norm{z_t}_H^2 = \scp{\opAn(u_t) - \opAn(v_t) + B_n(u_t) - B_n(v_t)}{z_t}_H \dt
\end{equation}
and the noise term cancels due to the fact that we consider strong solutions corresponding to the same noise realization, i.\,e. the same Brownian motion $W_n(t)$. The more or less standard procedure is to use \eqref{eq:DifferenceProcess} to establish a Gronwall argument. Due to the structure of the nonlinear drift that is in particular only locally monotone, such an argument involves some uniform (in $n$) exponential a priori estimates on the solution to \eqref{eq:ApproxSDE}. The first lemma concerns such an estimate that is relying on the invariance of the enstrophy similar to Proposition \ref{prop:ApproxMoment}.
\begin{lem}\label{lem:ExponentialEstimate}
There exists $\delta^\ast> 0$ inverse proportional to $\norm{\sqrt{Q}}_{L(V^\ast,H)}^2$ such that for all $0 < \delta \leq \delta^\ast$ and all $t \geq 0$ it holds that
\[
\EV{\exp \Bigl[ \delta \int_0^t (1 + \norm{\uns}_{2,p})^{2p-2} \ds \Bigr]} \leq \exp \Bigl[ \delta C \big(1 + \norm{u_0}_V^p + t \tr (-A)Q\big) \Bigr]
\]
with a finite constant $C$ independent of $n$.
\end{lem}
\begin{proof}
Such an a priori estimate is essentially based on It\^o's formula applied to $\norm{u_t}_V^2$, cf. Proposition \ref{prop:ApproxMoment}, that yields an inequality involving a local martingale. In the exponential this clearly does not vanish after taking the expectation. However, recall that for a continuous local martingale $M$ with $M_0=0$ the processes
\[
Z_t^\alpha \df \exp \Big[ \alpha M_t - \tfrac{\alpha^2}{2}\langle M \rangle_t \Big]
\]
for $\alpha >0$ are again continuous local martingales w.\,r.\,t. the same filtration. Such a local martingale is always a supermartingale due to Fatou's Lemma, in particular $\EV{Z^\alpha_0} = 1$, hence $\EV{Z^\alpha_t} \leq 1$ for all $t\geq 0$ and $\alpha >0$. This suggests the following strategy.
\begin{equation}\label{proof:Martingale2}
\begin{split}
\EV{\exp \bigl[ \tfrac{\alpha}{2} M_t \bigr]} &= \EV{\exp \bigl[ \tfrac{\alpha}{2} M_t - \tfrac{\alpha^2}{4} \langle M \rangle_t \bigr] \exp \bigl[\tfrac{\alpha^2}{4} \langle M\rangle_t \bigr]}\\
&\leq \EV{ \exp \bigl[ \alpha M_t - \tfrac{\alpha^2}{2} \langle M\rangle_t \bigr]}^{\frac12} \EV{\exp \bigl[\tfrac{\alpha^2}{2}\langle M\rangle_t \bigr]}^\frac12 \leq \EV{\exp \bigl[\tfrac{\alpha^2}{2}\langle M\rangle_t \bigr]}^\frac12.
\end{split}
\end{equation}

In our case, $M_t$ is a stochastic integral whose quadratic variation is always a $L^2$-norm in time. This is not quite suitable for the case $p<2$ since the left hand side is some $L^p$-norm in time, cf. \eqref{proof:EstimateIp}. In order to achieve an estimate anyway, we use It\^o's formula for $f(\norm{u_t}_V^2)$ with the concave function $f(x) = (1-\lambda)^{-1} (1+x)^{1-\lambda}$ for $\lambda \in (0,1)$.
\begin{equation}\label{proof:Itof}
\begin{split}
f\big( \norm{u_t}_V^2 \big) &\leq f \big( \norm{x}_V^2\big) - 2 \nu_0 (p-1)\int_0^t \frac{\Ip(u_t)}{\big(1 + \norm{u_t}_V^2\big)^\lambda} \dt\\
&\quad+ 2 \int_0^t \frac{\scp{u_t}{\sqrt{Q_n} \dwns}_V}{\big(1 + \norm{u_t}_V^2\big)^\lambda} + \int_0^t \frac{\tr (-A) Q}{\big(1 + \norm{u_t}_V^2\big)^\lambda} \dt.
\end{split}
\end{equation}
Why does that give any advantage over the usual formula for the square? While on the left hand side we have morally an exponent $p-2\lambda$, the quadratic variation of the local martingale admits the exponent $2-4\lambda$ due to the square. Notice that for $2\lambda \geq 2-p$ the former is larger, thus we can estimate the quadratic variation by the left hand side. Note that by Lemma \ref{app:EstIp}, $\Ip$ is directly related to an $L^p$-norm of the second derivative so that the continuous embeddings $H^{2,p} \hookrightarrow V \hookrightarrow H^{1,p}$, $p \in (1,2)$ imply
\begin{equation}\label{proof:EstimateIp2}
\frac{\Ip(u_t)}{\bigl(1 + \norm{u_t}_V^2\bigr)^\lambda} \geq C (1 + \norm{u_t}_{2,p})^{p-2\lambda} - C
\end{equation}
with a constant $C$ composed of the embedding constants only depending on $p$. In the following $C$ may change from line to line. With inequalities \eqref{proof:Itof} and \eqref{proof:EstimateIp2} and some obvious estimates follows
\[
\delta \int_0^t \big(1 + \norm{u_s}_{2,p}\big)^{p-2\lambda} \ds\leq \delta C \bigl(1 + \norm{x}_V^2\bigr)^{1-\lambda} + \delta \bigl( 1 + C \tr (-A) Q\bigr) t + \delta C M_n(t).
\]
Now introduce the stopping time $\tau_M \df \{ \inf t \geq 0: \int_0^t \Ip(u_s) \ds \geq M\}$. Clearly, $\tau_M \to \infty$ as $M\to \infty$ by \eqref{eq:APrioriEstimate}. An upper bound for the quadratic variation of the local martingale $M_n(t)$ is given by
\begin{equation}\label{proof:QuadVar}
\begin{split}
\langle M_n\rangle_t &= \int_0^t \frac{\norm{\sqrt{Q_n}(-A) u_s}_H^2}{\bigl(1 + \norm{u_s}_V^2\bigr)^{2\lambda}} \ds \leq \norm{\sqrt{Q}}^2_{L(V^\ast, H)} \int_0^t (1 + \norm{u_s}_V)^{2-4\lambda}\ds\\
&\leq C \norm{\sqrt{Q}}^2_{L(V^\ast, H)} \int_0^t (1 + \norm{u_s}_{2,p})^{2-4\lambda} \ds,
\end{split}
\end{equation}
where we used the embedding $H^{2,p} \hookrightarrow V$ again. With the choice $2 \lambda = 2 -  p$, the bound on the quadratic variation is exactly given by the lower bound for $\Ip$, hence with \eqref{proof:Martingale2}, \eqref{proof:EstimateIp2} and \eqref{proof:QuadVar} we have shown
\begin{align*}
\EV{\exp \Big[\delta \int_0^{t \wedge \tau_M} \big(1 + \norm{u_s}_{2,p}\big)^{2p-2} \ds\Big]}^2 &\leq \exp \Big[ \delta C \big( 1 + \norm{x}_V^p + (t \wedge \tau_M) \tr (-A) Q\Big)\Big]\\
&\quad+ \EV{\exp \Big[ \delta^2 (\delta^\ast)^{-1} \int_0^{t \wedge \tau_M} \big(1 + \norm{u_s}_{2,p}\big)^{2p-2} \ds\Big]}
\end{align*}
The right hand side is finite, thus the left hand side can absorb it provided $\delta \leq \delta^\ast$. Fatou's lemma yields the conclusion.
\end{proof}
We are now in the position to prove a pointwise gradient estimate for the transition semigroup $P_t^n$. It is based on the monotone behavior of the extra stress tensor, and thus the operator $\opA$, shown in Lemma \ref{app:Monotone}. Note that for $p<2$ this property is essentially weaker than in the case $p= 2$, the Navier-Stokes equations.
\begin{prop}\label{prop:EstTransSemi}
Suppose Assumption \ref{assum:LowerBound} holds. Then, for any $\phi \in \lip (H_n)$, $x \in H_n$, $t > 0$ and arbitrary $\delta > 0$ there exists a constant $C(\delta, Q, p, \nu_0)$ independent of $n$ such that the gradient of the transition semigroup $P_t^n$ satisfies
\begin{equation}\label{eq:EstGradTransSemi2}
\norm{D P_t^n \phi (x)}_H \leq \norm{\phi}_{\lip} \mathrm{e}^{\delta \norm{x}_V^p + C(\delta, Q, p, \nu_0)t}.
\end{equation}
\end{prop}
\begin{proof}
Consider \eqref{eq:DifferenceProcess} and split apart the right hand side into one part for the convection term $I_t \df \scp{B_n(u_t) - B_n(v_t)}{z_t}$ and one for the diffusion term $J_t \df \scp{\opAn(u_t) - \opAn(v_t)}{z_t}$. At first, we want to deal with $J_t$. By Lemma \ref{app:Monotone} and the inverse H\"older inequality with $3-p > r>1$ follows
\begin{align}
J_t & \leq  - C \int \big(1 + \abs{\mathbf{E}u_t}^2 + \abs{\mathbf{E}v_t}^2\big)^{\frac{p-2}{2}} \abs{\mathbf{E}z_t}^2 \dxi \nonumber\\
&\leq - C \Bigl( \int \abs{\mathbf{E} z_t}^{\frac{2}{r}} \dxi \Bigr)^r \Bigl( \int \big(1 + \abs{\mathbf{E} u_t}^2 + \abs{\mathbf{E} v_t}^2\big)^{\frac{2-p}{2(r-1)}} \dxi \Bigr)^{1-r} \nonumber\\
&\leq - C_1 \norm{z_t}_{1,\frac{2}{r}}^2 \Bigl(1 + \norm{u_t}_{1,\frac{2-p}{r-1}}^{2-p} + \norm{v_t}_{1,\frac{2-p}{r-1}}^{2-p} \Bigr)^{-1}.\label{proof:Jt}
\end{align}
Parts of the convection term are absorbed by this negative term. Note that if $p=2$ the estimate for $J_t$ would be significantly easier to handle and the well-known standard estimates for the convection term would apply. We use its trilinear form to obtain $I_t =  -\scp{z_t}{B(z_t,u_t)}$, as well as $I_t = \scp{z_t}{B(z_t,v_t)}$. Lemma \ref{app:ConvectionHoelder} yields the estimate
\[
\abs{I_t} \leq \frac12 \norm{z_t}_{0,q}^2 \big(1 + \norm{u_t}_{1, \frac{2-p}{r-1}} + \norm{v_t}_{1, \frac{2-p}{r-1}}\big)
\]
with the conjugated exponent $q = 2(2-p)/(3-p-r) > 1$. The Gronwall argument we are establishing in the following involves Lemma \ref{lem:ExponentialEstimate}, in particular the norm in $H^{2,p}$. Thus, we choose $r$ such that
\[
\frac{2-p}{r-1} = \frac{2p}{2-p} \quad\text{which is}\quad r = 1 + \frac{(2-p)^2}{2p} > 1.
\]
With this choice the embedding $H^{2,p} \hookrightarrow H^{1,\frac{2p}{(2-p)}}$ is continuous. Moreover, $q = 4p/(3p-2)$ and $4/3 < 2/r < 2$. Hence the embedding $H^{1,\frac{2}{r}} \hookrightarrow L^{q^\ast}$, $q^\ast \df 4p/(2-p)^2$ is continuous. Obviously, for $p \in (1,2)$ we have that $q \in (2,4)$ and $q^\ast \in (4, \infty)$. This allows us to use an interpolation between $L^2$ and $L^{q^\ast}$, with the parameter $\theta$ given by
\[
\frac{1}{q} = \frac{3p-2}{4p} = \frac{\theta}{2} + \frac{(1-\theta)(2-p)^2}{4p} \quad\text{implying}\quad \theta = \frac{3p-2-(2-p)^2}{2p-(2-p)^2}.
\]
Therefore we get
\begin{align}
\abs{I_t} &\leq C \norm{z_t}_H^{2\theta} \norm{z_t}_{1,\frac{2}{r}}^{2(1-\theta)} \big( 1 +  \norm{u_t}_{1, \frac{2p}{2-p}} + \norm{v_t}_{1, \frac{2p}{2-p}}\big)\nonumber\\ 
\begin{split}
&\leq C_1 \norm{z_t}_{1,\frac{2}{r}}^2 \Bigl(1 + \norm{u_t}_{1,\frac{2-p}{r-1}}^{2-p} + \norm{v_t}_{1,\frac{2-p}{r-1}}^{2-p} \Bigr)^{-1}\\ 
& \quad+ C \norm{z_t}_H^2 \big( 1 +  \norm{u_t}_{2,p} + \norm{v_t}_{2,p}\big)^\beta.
\end{split}
\label{proof:It}
\end{align}
The exponent $\beta$ is given by
\[
\beta = \frac{1 + (2-p)(1-\theta)}{\theta} = \frac{2p}{3p-2-(2-p)^2}.
\]
Altogether with \eqref{eq:DifferenceProcess}, \eqref{proof:Jt} and \eqref{proof:It} we derived the differential inequality
\[
\frac12 \ddt \norm{z_t}_H^2 \leq C \norm{z_t}_H^2 ( 1 +  \norm{u_t}_{2,p} + \norm{v_t}_{2,p})^\beta.
\]
Gronwall's lemma implies
\begin{equation}\label{proof:Gronwallzt}
\EV{\norm{z_t}_H^2} \leq \EV{ \exp \Bigl[C \int_0^t ( 1 +  \norm{u_s}_{2,p} + \norm{v_s}_{2,p})^\beta \ds \Bigr]} \norm{x-y}_H^2,
\end{equation}
provided $2p - 2 > \beta$. This is where Assumption \ref{assum:LowerBound} comes into play, since this condition on $p$ is equivalent to $p^3 - 8p^2+14p-6< 0$. The only relevant root of this polynomial is $p^\ast \approx 1.60407$. Thus, in the case $p \in (p^\ast,2)$ the expectation on the right hand side is finite due to Lemma \ref{lem:ExponentialEstimate}. Young's inequality implies for arbitrary $\delta > 0$
\[
\EV{\norm{z_t}_H^2} \leq \exp \Bigl[ \delta (1 + \norm{x}_V + \norm{y}_V)^p + C(\delta, Q, p, \nu_0) t \Bigr] 
\]
with a constant $C(\delta, Q, p, \nu_0)$ independent of $n$. Now, choose $y = x + h$, $h \in H_n$ with $\norm{h}_H \to 0$ in \eqref{eq:EstDiffTransSemi} to conclude the result.
\end{proof}
\begin{proof}[Proof of Theorem \ref{thm:MainTheorem}]
Suppose that $\psi \in L^2(\mu)$ is chosen such that the dual pairing with $(\lambda - K) \phi$ vanishes, i.\,e.
\begin{equation}\label{proof:Ortogonal}
\int (\lambda - K) \phi(x) \psi(x) \dmu(x) = 0 \quad \text{for all } \phi \in \fcb.
\end{equation}
If this property implies $\psi = 0$, the range condition \eqref{eq:RangeCondition} follows immediately. To show this, fix a function $\phi \in \fcb$ and choose $n$ such that $\phi$ can be identified with a function in $C_b^2(\R^n)$. First of all
\begin{align*}
&\ddt \mathrm{e}^{-\lambda t} \int P_t^n \phi \psi \dmu(x) = - \mathrm{e}^{-\lambda t} \int (\lambda - K_n) P_t^n \phi \psi \dmu(x)\\
&\quad= - \mathrm{e}^{-\lambda t} \int \scp{\opA(x) - \opAn(x) + B(x) - B_n(x)}{D P_t^n \phi} \psi \dmu(x)
\end{align*}
with the help of \eqref{proof:Ortogonal}. An integration w.\,r.\,t. $t$ yields
\begin{align*}
&\mathrm{e}^{-\lambda t} \int P_t^n \phi(x) \psi(x) \dmu(x) = \int \phi(x) \psi(x) \dmu (x)\\
&\qquad- \int_0^t \mathrm{e}^{-\lambda s} \int \scp{\opA(x) - \opAn(x) + B(x) - B_n(x)}{D P_s^n \phi(x)} \psi(x) \dmu (x) \ds.
\end{align*}
Since $\psi \in L^2(\mu) \subset L^1(\mu)$ and $P_T^n (C_b(H_n)) \subseteq C_b(H_n)$ we can estimate as follows.
\begin{align*}
&\left| \int \phi(x) \psi(x) \dmu (x) \right| \leq \mathrm{e}^{-\lambda t}C \norm{\phi}_{L^\infty(\mu)} \norm{\psi}_{L^1(\mu)}\\
&\qquad+ \int_0^t \mathrm{e}^{-\lambda s} \norm{\scp{\opA - \opAn + B - B_n}{D P_s^n \phi}}_{L^2(\mu)} \norm{\psi}_{L^2(\mu)} \ds.
\end{align*}

\textbf{Claim.} It holds that $\lim_{n \to \infty} \scp{\opA - \opAn + B - B_n}{D P_s^n \phi} = 0$ in $L^2(\mu)$.
 
Once this claim is proven, it follows that $\psi = 0$ $\mu$-a.\,s. by letting $t \to \infty$. Hence, the range is dense in $L^2(\mu)$. Let us now prove the claim. H\"older's inequality yields the estimate
\begin{align*}
&\scp{\opA(x) - \opAn(x) + B(x) - B_n(x)}{D P_s^n \phi(x)}\\
&\quad\leq \bigl(\norm{\opA(x) - \opAn(x)}_H + \norm{B(x) - B_n(x)}_H\bigr) \norm{D P_s^n \phi(x)}_H.
\end{align*}
Clearly, $\norm{\opA(x) - \opAn(x)}_H \to 0$ for all $x \in D(\opA)$ and $\norm{B(x) - B_n(x)}_H \to 0$ for all $x \in D(B)$. Thus, it suffices to show that
\[
\int \bigl(\norm{\opA(x)}_H + \norm{B(x)}_H\bigr)^2 \norm{D P_s^n \phi(x)}_H^2 \dmu (x) < \infty,
\]
which implies the claim by dominated convergence. By Lemmas \ref{app:NormA} and \ref{app:NormB} we know that
\[
\norm{\opA(x)}_H + \norm{B(x)}_H \leq C \big( 1 + \norm{x}_V \big)^{\frac{4-p}{2}} \Ip(x)^{\frac{1}{2}}.
\]
Proposition \ref{prop:EstTransSemi} now implies that
\begin{align*}
&\int \bigl(\norm{\opA(x)}_H + \norm{B(x)}_H\bigr)^2 \norm{D P_s^n \phi(x)}_H^2 \dmu (x)\\
&\quad\leq C \norm{\phi}_{\lip}^2 \int ( 1 + \norm{x}_V )^{4-p} \Ip(x) \mathrm{e}^{ \delta \norm{x}_V^p + C(\delta, Q,p,\nu_0)s} \dmu(x),
\end{align*}
which is finite by Theorem \ref{thm:Measure} for $\delta > 0$ sufficiently small.
\end{proof}

\renewcommand{\thesection}{\Alph{section}}
\setcounter{section}{0}
\section{Appendix}
In this section we include the necessary technical lemmas concerning the properties of the nonlinear drift of equation \eqref{eq:SPDE} that can be found in the relevant literature concerning generalized Newtonian fluids, see e.\,g. \cite{NecasBook}.
\subsection{Properties of the Operator $\opA$}
At first, introduce a crucial lemma concerning the equivalence of the gradient and its symmetric part, for a proof we refer to \cite{NecasBook}.
\begin{lem}[Korn's Lemma]\label{app:Korn}
There exists a constant $K_p$ depending on $p$ (and the domain) such that $\norm{\nabla u}_{0,p} \leq K_p \norm{\symg}_{0,p}$.
\end{lem}
\begin{lem}\label{app:DefExtraStress}
Let $\mathbf{S}$ be any symmetric extra stress tensor. Then for $u,v \in \mathcal{D}_\sigma$ it holds that
\[
\int_{\To^2} \Div \big(\mathbf{S}(\symg)\big) v \dxi = -\sum_{ij} \int_{\To^2} \mathbf{S}_{ij}(\symg) (\mathbf{E}v)_{ij} \dxi.
\]
\end{lem}
\begin{proof}
This is simply an integration by parts and the symmetry of $\mathbf{S}$.
\end{proof}
In all the following suppose that Assumption \ref{assum:PowerLaw} holds. The next lemma can be found in e.\,g. \cite[Theorem 4]{DieningRuzicka}.
\begin{lem}\label{app:Monotone}
There exists a constant $C$ such that for $u,v \in \mathcal{D}_\sigma$
\[
\scp{\opA(u) - \opA(v)}{u-v}_H \leq - C \int_{\To^2} \big( 1 + \abs{\symg}^2 + \abs{\mathbf{E}v}^2\big)^{\frac{p-2}{2}} \abs{\symg - \mathbf{E}v}^2 \dxi.
\]
\end{lem}

\begin{lem}\label{app:TestStokes}
For $u \in \mathcal{D}_\sigma$ it holds that $\scp{\opA(u)}{(-A)u}_H \leq - \nu_0(p-1) \Ip(u)$.
\end{lem}
\begin{proof}
With Lemma \ref{app:DefExtraStress} and $Au = \Delta u$ because of the periodicity, we know that
\begin{align*}
&\scp{\opA(u)}{(-A)u}_H = -\nu_0 \sum_{ij} \int_{\To^2} \big(1 + \abs{\symg}^2\big)^{\frac{p-2}{2}}(\symg)_{ij}(\mathbf{E} (-\Delta) v)_{ij} \dxi.\\
\intertext{A permutation of the derivatives and an integration by parts yield}
&\quad= - \nu_0 \int_{\To^2} \big( 1 + \abs{\symg}^2\big)^{\frac{p-2}{2}} \sum_{ijm} \big(\partial_m (\symg)_{ij} \big)^2 \dxi\\
&\quad\quad - \nu_0 \int_{\To^2} \big( 1 + \abs{\symg}^2\big)^{\frac{p-4}{2}} \sum_{ijklm}  (p-2) (\symg)_{kl} \partial_m (\symg)_{kl}(\symg)_{ij}\partial_m (\symg)_{ij} \dxi\\
&\quad\leq -\nu_0 \int_{\To^2} \big( 1 + \abs{\symg}^2\big)^{\frac{p-4}{2}} \big((p-2) \abs{\symg}^2 + 1 + \abs{\symg}^2\big) \abs{\nabla \symg}^2 \dxi \leq \nu_0 (p-1) \Ip(u).\qedhere
\end{align*}
\end{proof}
The following lemma yields estimates from below of the term $\Ip(u)$ by means of first and second order derivatives in the case $p \in (1,2)$, see \cite[Lemma 5.3.24]{NecasBook} for the general one.
\begin{lem}\label{app:EstIp}
$\norm{u}_{2,p}^2 \leq C\, \Ip(u) (1 + \norm{u}_{1,p})^{2-p}$ for all $u \in \mathcal{D}_\sigma$ with a constant $C$ depending only on $p$ (and the domain).
\end{lem}
\begin{lem}\label{app:NormA}
For $u \in \mathcal{D}_\sigma$ follows $\norm{\opA(u)}_H^2 \leq \nu_0 \Ip(u)$. 
\end{lem}
\begin{proof}
At first, consider the following property of the extra stress tensor $\mathbf{S}$. In the following, $\partial_{ij}$ stands for the derivative w.\,r.\,t. the matrix component $ij$. For $A, B \in \R^{2\times 2}_{\text{sym}}$ we know that
\begin{align*}
\sum_{ijkl} \partial_{ij} \mathbf{S}_{kl}(A) B_{ij} B_{kl} &= \nu_0 \sum_{ijkl} (p-2) \big(1 + \abs{A}^2\big)^{\frac{p-4}{2}} A_{ij} B_{ij} A_{kl} B_{kl} + \big(1 + \abs{A}^2\big)^{\frac{p-2}{2}} \delta_{ij,kl} B_{ij} B_{kl}\\
&= \nu_0 (p-2) \big(1 + \abs{A}^2\big)^{\frac{p-4}{2}} \Big( \sum_{ij} A_{ij} B_{ij} \Big)^2 + \nu_0 \big(1 + \abs{A}^2\big)^{\frac{p-2}{2}} \abs{B}^2\\
&\leq \nu_0 \big(1 + \abs{A}^2\big)^{\frac{p-2}{2}} \abs{B}^2,
\end{align*}
since $p-2<0$. The matrix $B$ was arbitrary, hence
\[
\Big( \sum_{ijkl} \big(\partial_{ij} \mathbf{S}_{kl}(A)\big)^2\Big)^\frac12 \leq \nu_0 \big(1 + \abs{A}^2\big)^{\frac{p-2}{2}}
\]
for all $A \in \R^{2\times 2}_{\text{sym}}$. With this follows immediately
\begin{align*}
\norm{\opA(u)}_H^2 &= \int_{\To^2} \sum_{k} \Big( \sum_l \partial_j \mathbf{S}_{kl}(\symg)\Big)^2 \dxi \leq 8 \int_{\To^2} \sum_{ijkl} \big( \partial_{ij} \mathbf{S}_{kl}(\symg) \partial_j (\symg)_{ij}\big)^2 \dxi\\
& \leq 8 \nu_0^2 \int_{\To^2} \big( 1 + \abs{\symg}^2\big)^{p-2} \abs{\nabla \symg}^2 \dxi \leq 8 \nu_0^2 \Ip(u).\qedhere
\end{align*}
\end{proof}

\subsection{Properties of the Convection Term $B$}
Elementary, but for sake of completeness stated here, are the properties of the convection term $(u \cdot \nabla) u$ and its associated trilinear form.
\begin{lem}\label{app:ConvectionInvariance}
For $u,v,w \in \mathcal{D}_\sigma$ it holds that
\[
\scp{B(u,v)}{v}_H = 0 \quad \text{and}\quad \scp{B(u,v)}{w}_H = - \scp{B(u,w)}{v}_H.
\]
\end{lem}
\begin{proof}
With an integration by parts we have
\begin{gather*}
\sum_{ij} \int_{\To^2} u_i (\partial_i v_j) v_j \dxi = \frac12 \sum_{ij} \int_{\To^2} u_i \partial_i v_j^2 \dxi = - \frac12 \int_{\To^2} \Div u \abs{v}^2 \dxi = 0.
\end{gather*}
The second statement follows with $v$ replaced by $v+w$.
\end{proof}
\begin{lem}\label{app:ConvectionHoelder}
For $u,v,w \in \mathcal{D}_\sigma$ and $p_i \geq 1$ with $1/p_1 + 1/p_2 + 1/p_3 = 1$ it holds that
\[
\abs{\scp{B(u,v)}{w}_H} \leq \snorm{u}{0, p_1} \snorm{v}{1, p_2} \snorm{w}{0,p_3}.
\]
\end{lem}
\begin{proof}
H\"older's inequality and $\snorm{\nabla v}{0,p} \leq \snorm{v}{1,p}$ for all $p \geq 1$.
\end{proof}
\begin{cor}\label{app:NormB}
For $u \in \mathcal{D}_\sigma$ follows $\norm{B(u)}_H^2 \leq C (1 + \norm{u}_V)^{\frac{4-p}{2}} \Ip(u)^\frac12$.
\end{cor}
\begin{proof}
This follows from Lemma \ref{app:ConvectionHoelder} with $p_1 = 2p/(2p-2)$, $p_2 = 2p/(2-p)$ and $p_3 =2$. With this choice the embeddings $V \hookrightarrow L^{p_1}$ and $H^{2,p} \hookrightarrow H^{1,p_2}$ are continuous and we can apply Lemma \ref{app:EstIp}.
\end{proof}
\begin{lem}\label{app:ConvectionStokes}
For $u \in \mathcal{D}_\sigma$ it holds that $\scp{B(u)}{(-A) u}_H = 0$.
\end{lem}
\begin{proof}
Because of the periodicity we know that $Au = \Delta u$. Thus, an integration by parts implies
\begin{align*}
\scp{B(u,u)}{(-A)u}_H &= \sum_{ij} \int_{\To^2} u_i (\partial_i u_j) (-\Delta u)_j \dxi \\
&= \sum_{ijk} \int_{\To^2} \partial_k u_i \partial_i u_j \partial_k u_j + u_i \partial_k \partial_i u_j \partial_k u_j \dxi\\
&= \sum_{ijk} \int_{\To^2} \partial_k u_i \partial_i u_j \partial_k u_j \dxi - \frac12 \sum_{jk} \int_{\To^2} \Div u (\partial_k u_j)^2 \dxi = 0
\end{align*}
since $\Div u =0$ and therefore $\partial_1 u_1 = - \partial_2 u_2$. A comparison of all summands shows that the first sum vanishes, too.
\end{proof}
\section*{Acknowledgment}
During the preparation of this article the author was supported by the DFG and JSPS as a member of the International Research Training Group Darmstadt-Tokyo IRTG 1529.

\end{document}